\documentclass[12pt,a4paper]{amsart}
\usepackage{amsmath}
\usepackage{amssymb}
\usepackage{latexsym}
\usepackage{xypic}
\input xy
\xyoption{all}

\newtheorem{theorem}{Theorem}[section]
\newtheorem{lemma}[theorem]{Lemma}
\newtheorem{proposition}[theorem]{Proposition}

\newtheorem{corollary}[theorem]{Corollary}
\newtheorem{conjecture}[theorem]{Conjecture}
\newtheorem{exmp}[theorem]{Example}
\newtheorem{exmps}[theorem]{Examples}
\newtheorem{rem}[theorem]{Remark}

\newenvironment{remark}{\begin{rem}\rm}{\end{rem}\rm}

\newcommand{\qqed}{\hspace*{\fill}$\Box$}

\newcommand{\beeq}[1]{\begin{eqnarray}\label{#1}}
\newcommand{\eneq}{\end{eqnarray}}
\newcommand{\cal}{\mathcal}

\newcommand{\kc}{{\cal C}}

\newcommand{\kl}{{\cal L}}
\newcommand{\kn}{{\cal N}}
\newcommand{\km}{{\cal M}}
\newcommand{\ko}{{\cal O}}

\newcommand{\kt}{{\cal T}}
\newcommand{\kx}{{\cal X}}

\newcommand{\ZZ}{\mathbb{Z}}

\newcommand{\CC}{\mathbb{C}}

\newcommand{\PP}{\mathbb{P}}

\newcommand{\gm}{{\mathfrak M}}

\newcommand{\CH}{{\rm CH}}

\newcommand{\Pic}{{\rm Pic}}

\newcommand{\NS}{{\rm NS}}

\newcommand{\id}{{\rm id}}

\newcommand{\verylongarrow}[1]{\hbox to #1{\rightarrowfill}}
\newcommand{\congpf}{\xymatrix@1@=15pt{\ar[r]^-\sim&}}
\renewcommand{\to}{\xymatrix@1@=15pt{\ar[r]&}}

\begin{document}

\title[Stable maps and  Chow groups]{Stable maps
and Chow groups}

\author[D.\ Huybrechts and M.\ Kemeny]{D.\ Huybrechts and M.\  Kemeny}
\address{{\rm Daniel Huybrechts}\\
Mathematisches Institut, Universit{\"a}t Bonn\\
 Endenicher Allee  60\\
 53115 Bonn\\
 Germany}
\email{huybrech@math.uni-bonn.de}

\address{{\rm Michael Kemeny}\\
Mathematisches Institut, Universit{\"a}t Bonn\\
Endenicher Allee 60\\
 53115 Bonn\\
 Germany}
\email{michael.kemeny@gmail.com}
\thanks{This work was supported by the SFB/TR 45 `Periods, Moduli Spaces and Arithmetic of Algebraic Varieties' of the DFG (German Research Foundation).
The second author is supported by  PhD-scholarship of the Bonn International Graduate
School in Mathematics.} 

\begin{abstract}
According to the Bloch--Beilinson conjectures, an automorphism of a K3 surface $X$ that
acts as the identity on the transcendental lattice should act trivially on $\CH^2(X)$.
We discuss this conjecture for symplectic involutions and prove it in one third
of all cases. The main point is to use special elliptic K3 surfaces and stable maps
to produce covering families of elliptic curves on the generic K3 surface that are invariant under the involution.
\end{abstract}
 \maketitle

\subsection{}
Let $X$ be a complex projective K3 surface with an automorphism $f:X\congpf X$. According to the general philosophy
of the Bloch--Beilinson conjectures, the induced action of $f$ on the kernel of the cycle map $\CH^*(X)\to H^*(X,\ZZ)$ should be determined by the action of $f$ on the cokernel of the cycle map. More precisely, one expects the following to be true:

\begin{conjecture}\label{conj:BB}
 $f^*=\id $ on $\CH^2(X)_0$ if and only if $f^*=\id$ on $T(X)$.
\end{conjecture}

Here, $\CH^2(X)_0\subset\CH^2(X)$ is the degree zero part, i.e.\ the kernel of 
the cycle map $\CH^2(X)\to H^4(X,\ZZ)\cong\ZZ$, and $T(X)\subset H^2(X,\ZZ)$ is the transcendental lattice which can be described as the orthogonal complement of the N\'eron--Severi group $\NS(X)\subset H^2(X,\ZZ)$.  Alternatively, $T(X)\subset H^2(X,\ZZ)$ is the smallest sub-Hodge structure such that $H^{2,0}(X)\subset T(X)\otimes\CC$.
Thus, $f^*=\id$ on $T(X)$ if and only if $f$ acts trivially on $H^{2,0}(X)$. The latter is spanned by the unique (up to scaling) regular two-form $\sigma\in H^0(X,\Omega_X^2)$,
of which we think as a holomorphic symplectic structure. For this reason, an automorphism $f:X\congpf X$ with $f^*=\id$ on $T(X)$ is called a symplectomorphism.

It is well known that  $f^*=\id$ on $\CH^2(X)_0$ implies that $f$ acts trivially on $T(X)$
(see e.g.\ \cite[Ch.\ 23]{Voisin}). Appropriately rephrased, this holds for arbitrary smooth projective varieties and for arbitrary correspondences. It is the converse of the statement that is difficult  and that shall  be discussed here for symplectic  involutions of K3 surfaces, i.e.\ automorphisms $f$ of order two with $f^*\sigma=\sigma$.
 
 \subsection{} K3 surfaces $X$ endowed with a symplectic involution $f:X\congpf X$ come in families. As shown by van Geemen and Sarti in
 \cite{vGS}, the moduli space of such $(X,f)$ has one resp.\ two connected components in each
 degree $2d>0$ depending on the parity of $d$.
 To be more precise, let $\Lambda_d$ be the lattice $\ZZ\ell\oplus E_8(-2)$ with $(\ell.\ell)=2d$
 and denote for $d\equiv0(2)$ by $\widetilde\Lambda_d$ the unique even lattice containing $\Lambda_d$
with $\widetilde\Lambda_d/\Lambda_2\cong\ZZ/2\ZZ$ and such that $E_8(-2)\subset\widetilde\Lambda_d$ is primitive (see \cite[Prop.\ 2.2]{vGS}).
 
Then for generic $(X,f)$ one has $\NS(X)\cong\Lambda_d$ or $\NS(X)\cong\widetilde\Lambda_d$.
The class $\ell$ corresponds under this isomorphism to an ample line bundle $L$ on $X$ which spans the $f$-invariant part of $\NS(X)$. If $(X,f)$ is not generic, then one still finds
$\Lambda_d$ or $\widetilde\Lambda_d$ as a primitive sublattice in $\NS(X)$ with $E_8(-2)$ as the
orthogonal complement of the invariant part. Conversely, by the Global Torelli theorem
any $X$ parametrized by the (non-empty and in fact $11$-dimensional) connected moduli
spaces $\gm_{\Lambda_d}$ or $\gm_{\widetilde\Lambda_d}$ of $\Lambda_d$ resp.\ $\widetilde\Lambda_d$-lattice polarized K3 surfaces comes with a symplectic involution $f$ that is determined by
its action $=-\id$ on $E_8(-2)$ and $=\id$ on its orthogonal complement.
 
In other words, for each $d\equiv1(2)$ the moduli space of
K3 surfaces $X$ with a symplectic involution $f$ and an invariant
 polarization of degree $2d$ has one connected component $\gm_{\Lambda_d}$,
 whereas  for $d\equiv 0(2)$ it has two connected components, $\gm_{\Lambda_d}$
 and $\gm_{\widetilde\Lambda_d}$. Thus the following theorem, the main result of the
 present paper, proves Conjecture \ref{conj:BB} in one third of all possible cases.
 
 \begin{theorem}\label{thm:main}
 Let $d\equiv0(2)$ and $(X,f)\in\gm_{\widetilde\Lambda_d}$.
 Then $f^*=\id$ on $\CH^2(X)$.
 \end{theorem}

 \subsection{}
For $d=1$ (double covers of $\PP^2$) and $d=2$ (quartics in $\PP^3$) the conjecture is known to hold, see \cite{Chatz,Ped,Voi}. For $d=3$ (complete intersection
of a cubic and a quadric in $\PP^4$) an interesting approach is outlined in \cite{GT}.
Theorem 3.2 in \cite{Voisin2} proves the conjecture for equivariant
complete intersections in varieties with trivial Chow groups.

In \cite{HuyMSRI} the conjecture has been proven for $(X,f)$
in dense subsets  of $\gm_{\Lambda_d}$ and $\gm_{\widetilde \Lambda_d}$.
The proof there relies on Fourier--Mukai equivalences of the bounded derived category
of coherent sheaves on $X$ and it is not clear how to push the techniques further to cover
generic and hence arbitrary $(X,f)$.

The techniques to prove Theorem \ref{thm:main} can be applied to
symplectic automorphisms $f:X\congpf X$ of order $>2$. If the order of
$f$ is a prime $p$, then $p=2,3,5$, or $7$ (cf.\ \cite{Nik}), and  the results of \cite{vGS}
have in \cite{GS} successfully be generalized to cover also the cases
$p=3,5$, and $7$. Our methods prove Conjecture \ref{conj:BB} for many components
of the moduli space of $(X,f)$ in these cases. For a few more comments see Section \ref{sec:p357}.

%%%%%%%%%%%%%%%%%%%%%%%%%%%%%%%%%
\subsection{}

The proof of Theorem \ref{thm:main}  neither uses derived categories as
in  \cite{HuyMSRI} nor any deep cycle arguments as e.g.\ in \cite{GT}.
As we shall explain in Section \ref{sec:elliptic}, it is enough to find
a dominating family of integral genus one curves on $X$ 
that  are invariant under $f$ and avoid the fixed points of $f$. The conjecture is then
deduced from the absence of torsion in $\CH^2(X)$. It is not clear whether
the existence of such a family should be expected in general, but
it  will be shown here for generic K3 surfaces parametrized by
points in $\gm_{\widetilde\Lambda_d}$. This is done in two steps. Firstly,
we construct a family of genus one (reducible)  curves on a particular elliptic K3 surface for which $f$ 
is given by translation by a two-torsion section, see Section \ref{sec:specialell}. Then, the
theory of stable maps is applied to obtain the desired family for generic $X$.

The missing piece to prove Conjecture \ref{conj:BB} in full generality,
or at least for symplectic involutions, is the lack of special K3 surfaces  in $\gm_{\Lambda_d}$
for which appropriate families of genus one curves can be described explicitly.

{\bf Acknowledgments:} We thank Richard Thomas for a useful discussion
concerning  Section \ref{sec:defo}.

%%%%%%%%%%%%%%%%%%%%%%%%%%%%%%%%%%% 
\section{Covering families of elliptic curves}\label{sec:elliptic}
 
 Let $f:X\congpf X$ be a symplectic automorphism of finite order and denote
 its quotient by $\bar X:=X/\langle f\rangle$, which is a singular K3 surface. Moreover, if $f$ has prime order $p$,  then $p=2,3,5,$ or $7$ (see \cite{Nik}). In order to prove Conjecture \ref{conj:BB} for symplectic automorphisms of finite order (and we do not have anything to say for automorphisms of infinite order), one can restrict to those.
 The number of fixed points of $f$, all isolated, can be
 determined by the Lefschetz fixed point formula. E.g.\ for symplectic involution, i.e.\
 $p=2$, there are exactly eight fixed points.    
 
In the following, a family $\kc_t\subset X$ of curves given by $\kc\subset S\times X$ is called dominating if
 the projection $\kc\to X$ is dominant, i.e.\ the curves $\kc_t$ parametrized by the closed point
 $t\in S$ cover a Zariski open subset of $X$.
 
 \begin{proposition}\label{prop:Roit}
 Let $f:X\congpf X$ be a symplectic automorphism. Assume there exists  a dominating
 family of integral $f$-invariant curves $\kc_t\subset $X of geometric genus one with
 $\kc_t\cap {\rm Fix}(f)=\emptyset$ for generic $t$. Then $f^*=\id$ on $\CH^2(X)$. 
 \end{proposition}

\begin{proof} It suffices to prove that for generic $x\in X$ the points $x$ and $y:=f(x)$ are rationally equivalent, i.e.\ $[x]=[y]$ in $\CH^2(X)$. Since by Roitman's theorem
 $\CH^2(X)$ is torsion free (see e.g.\ \cite[Ch.\ 22]{Voisin}), the latter
 is equivalent to $[x]-[y]$ being torsion. For any morphism $g:C\to X$  from a smooth irreducible curve $C$ the induced $g_*:\Pic(C)=\CH^1(C)\to \CH^2(X)$ is a group homomorphism. Thus, if there exist lifts $\tilde x,\tilde y\in C$ of $x$ resp.\ $y$  such that $\ko(\tilde x-\tilde y)\in\Pic^0(C)$ is a torsion line bundle, then automatically $[x]=[y]$ in $\CH^2(X)$.

By assumption, any generic closed point $x\in X$ lies on one of the curves $\kc_t$.
Since the curves $\kc_t$ are assumed to be $f$ invariant, $y=f(x)$ is contained in the same curve and $f$ lifts to an automorphism $\tilde f$ of the normalization
$C:=\tilde\kc_t$, which is a smooth integral curve of genus one.
As $\kc_t$ avoids the fixed points of $f$, the automorphism
$\tilde f:C\congpf C$ is fixed point free and hence $D:=C/\langle\tilde f\rangle$
is smooth of genus one, too. After choosing origins for $C$ and $D$ appropriately,
$C\to D$ is a morphism of elliptic curves which can be viewed as a quotient of
$C$ by a finite subgroup $\Gamma\subset C\cong\Pic^0(C)$. Hence, points in the same fibre of $C\to D$ differ by elements of $\Gamma$. In particular, $\ko_C(\tilde x-\tilde y)\in\Gamma\subset\Pic^0(C)$ is a torsion line bundle.\end{proof}

The problem now becomes to construct a family of genus one curves as required. We do not
know how to do this directly. On the special elliptic surface considered
in  Section \ref{sec:specialell} a family of genus one curves is constructed, but the curves
are not integral. They become integral only after deformations to the generic case.

%\begin{corollary} Assume $X$ contains an ample integral curve $C\subset X$ which is %nodal rational and avoids the fixed points of a symplectic automorphism $f:X\congpf X$.
%If $C$ is invariant under $f$,  then $f^*=\id$ on $\CH^2(X)$.
%\end{corollary}
%
%\begin{proof} Consider the quotient curve $\bar C:=C/\langle f\rangle\subset \bar X$
%and its associated line bundle $\bar L:=\ko(\bar C)$. Then $\bar C$ is still nodal rational
%and $\bar L$ is ample. Let $g:=\dim|\bar L|$ and
%apply the following standard dimension count.
%In  $\bar M_g$ the closed subset of stable curves with at least $\delta$
%nodes is of pure  codimension $\delta$. Thus, its preimage under the rational map $|\bar L|%\to \bar M_g$ is empty or of codimension at most $\delta$.
%Applied to $\delta=g-1$, this shows that the nodal rational curve $\bar C$ sits in a one-%dimensional family of nodal curves
%$\bar\kc_t$ of geometric genus $1$. (We use that $X$ cannot be covered by rational %curves.) The pre-image $\kc_t$ of $\bar \kc_t$ under $X\to \bar X$ is clearly $f$ invariant. %Since $C$ was integral, also the generic $\kc_t$ is integral. Similarly,  $\kc_t\cap{\rm Fix}%(f)=\emptyset$ for generic $t$. Thus, the proposition
%applies.\end{proof}

 %%%%%%%%%%%%%%%%%%%%%%%%%%%%%%%%%%%
\section{Stable maps to K3 surfaces}\label{sec:defo}
 
 Let $\kx\to S$ be an irreducible family of K3 surfaces with a global line bundle $\kl$.
 Consider  the moduli stack $\km_g(\kx,\kl)\to S$ of stable maps $h:D\to \kx_t$
 to fibres of $\kx\to S$ such that $D$  is of arithmetic genus $g$ with $h_*(D)\in|\kl_t|$.
 The stack structure of $\km_g(\kx,\kl)$ is of no importance to us, so we shall
 ignore it and treat $\km_g(\kx,\kl)$ as a moduli space. If we do not want to fix the
 linear equivalence class of the image curves, we simply write $\km_g(\kx)$.
 
 The following fact has been used in various contexts in the literature, but mostly for $g=0$
 (see e.g.\ \cite{BHT,LL}). We shall need the following statement for $g=1$.
  
 \begin{proposition}\label{prop:defo}
 Every irreducible component of $\km_g(\kx,\kl)$ is of dimension at least
 $g+\dim (S)$.
 \end{proposition}
 
 \begin{proof}
 The starting point is \cite[Thm.\ 2.17]{Kol}: For simplicity let $\pi:\kx\to S$ be 
 a smooth projective family  over an irreducible base $S$ and
let ${\mathcal D}\to S$ be a flat and projective family of curves. 
Every irreducible component of ${\rm Mor}_S({\mathcal D},\kx)$ containing a
morphism $h:D:={\mathcal D}_0\to X:=\kx_0$  is of dimension at least
 \begin{equation}\label{eqn:Kollar}
 \chi(D,h^*\kt_ X)+\dim (S).
 \end{equation}
 The first term $\chi(D,h^*\kt_ X)=h^0(D,h^*\kt_ X)-h^1(D,h^*\kt_ X)$ reflects the usual
 defor\-mation-obstruction theory for the morphism $h:D\to X$. A priori, the obstructions 
 to deform the morphism  $h:D\to\kx$ 
 are contained in $H^1(D,h^*\kt_\kx)$,
 which is part of an exact sequence $$\ldots\to H^1(D,h^*T_X)\to H^1(D,h^*\kt_\kx)\to H^1(D,h^*\pi^*\kt_S)\to0.$$
 Since the morphism $D\to X\subset \kx\to S$ is constant, there are no obstructions to deform
 it sideways at least when $S$ is smooth. In other words, the obstructions to deform
 $h:D\to\kx$ are contained in the image of $H^1(D,h^*\kt_X)$ which leads to the stronger
 bound in (\ref{eqn:Kollar}).
 
 A similar argument allows one to treat the case of varying domain $D$. The usual obstruction theory for stable maps shows that $\km_g(X,L)$ in $[h:D\to X]$ is of dimension
 at least $\chi(D,(h^*\Omega_X\to\Omega_D)^*)$, where the two term complex
 $h^*\Omega_X\to \Omega_D$ is concentrated in degree $-1$ and $0$, see
 \cite{GHS}. For $X=\kx_0$ in a family $\kx\to S$, the analogue
 of (\ref{eqn:Kollar}) then says that $\km_g(\kx)$ in a point corresponding to a stable map $h:D\to X$ is of dimension at least
 \begin{equation}\label{}
 \chi(D,(h^*\Omega_X\to\Omega_D)^*)+\dim (S)=g-1+\dim (S)
 \end{equation}  The last equation follows from a standard Riemann--Roch
 calculation.
 
 The remaining issue is to increase the bound by restricting to families $\kx\to S$
 which come with a deformation $\kl$ of $L:=\ko(h_*(D))$. 
 One can either evoke reduced deformation theory for K3 surfaces as developed
 recently in \cite{KT} in great detail or use the following trick.
 
 Any given family $(\kx,\kl)\to S$ with a polarization $\kl$ can be thickened to a family
 $\tilde\kx\to\tilde S$ with $\dim\tilde S=\dim S+1$ such that
 transversally to $S\subset\tilde S$ the line bundle $\kl$ is obstructed (even to first order).
 More precisely, for $t\in S$ the line bundle $\kl_t$ on $\kx_t$ deforms to first order in the
 direction of $v\in T_{\tilde S,t}$ if and only if $v\in T_{S,t}\subset T_{\tilde S,t}$.
 If $\kl$ is fibrewise ample, then the thickening $\tilde\kx\to\tilde S$ can be explicitly
 described by using the twistor space  construction for each fibre $\kx_t$ and the K\"ahler class  given by ${\rm c}_1(\kl_t)$. (Note that in particular, $\tilde \kx\to\tilde S$ will in general not be projective.) Otherwise, one uses the standard deformation theory of K3 surfaces to
 produce such a family at least locally, which is enough for the following dimension count.
 
 By the discussion above, $\km_g(\tilde\kx)$ is in $[h:D\to X]$ of dimension
 $$g-1+\dim(\tilde S)=g+\dim (S).$$ On the other hand, $h:D\to X$ cannot deform sideways
 in a tangent direction $v\in T_{\tilde S,0}$ that is not contained in $T_{S,0}$, because
  $\ko(h_*(D))=\kl_0$. This shows that the two moduli spaces
  $\km_g(\tilde \kx)$ and $\km_g(\kx)$ coincide near the point given by $[h:D\to X]$.
 %The main point here is the, by now
 %classical, observation that the obstructions to deform $h:D\to X$ (with $X$ a fixed K3 %surface)
% are contained in the kernel of the natural surjection $${\mathbb H}^2(D,(h^*\Omega_X\to%\Omega_D)^*)\to H^2(X,\ko_X).$$
% (Roughly, this map is obtained by taking global sections of the natural map
% $$h^*\omega_X\to h^*\Omega_X\otimes h^*\Omega_X \to h^*\Omega_X\otime\omega_D,$$ composing with $H^0(X,\omega_X)\to H^0(D,h^*\omega_X)$ and then dualizing.)
% 
% This immediately yields $\dim(\km_g(X),h)\geq \chi(D,(h^*\Omega_X\to\Omega_D)^*)+1= %g$ (c.f.\ \cite[Thm.\ 2.4]{KT}). 
% 
% This seems to suggest the lower bound we are after, but  one needs
% to argue that the reduced obstruction theory also works in the relative setting, i.e.\
% that  the obstructions are contained in the kernel
% of a surjection
% \begin{equation}\label{eqn:Pridham}
% {\mathbb H}^2(D,(h^*\Omega_\kx\to\Omega_D)^*)\to H^2(X,\ko_X)
% \end{equation}
% or, rather, in the kernel of the restriction to the image of
%$${\mathbb H}^2(D,(h^*\Omega_X\to\Omega_D)^*)\to{\mathbb H}^2(D,(h^*\Omega_\kx\to%\Omega_D)^*).$$
% The issue has been discussed in detail in \cite{KT} where it has been shown
% that the obstructions are contained in a space complementary to a
% section of (\ref{eqn:Pridham}). This
% is clearly enough for the dimension argument. The full statement can be deduced
% from  \cite{P}.
 \end{proof}

 \begin{corollary}\label{cor:defo}
 Suppose the fibre $\km_0$ of an irreducible component $\km\subset\km_g(\kx,\kl)$
 is of dimension $\leq g$ for some $0\in S$. Then $\km$ dominates $S$.\qqed
 \end{corollary}
 
In other words, if the moduli space $\km_g(\kx_0,\kl_0)$ of stable maps
to one fibre $\kx_0$ has the expected dimension $g$ in $[h:D\to \kx_0]$, then $h$ can be deformed to a stable map $h_t:D_t\to \kx_t$ to the generic fibre. To ensure that the condition is met, we shall later use the following criterion, c.f.\ \cite[Cor.\ 1.2.5]{Kem} and \cite[Lem.\ 2.6]{LL}.

\begin{proposition}\label{prop:defounramified}
 Suppose the stable map $h:D\to X$ 
satisfies the following conditions:\\
i) If  $D_1,D_2,\ldots,D_{n}$ are the components of $D$, then $D_2,\ldots,D_n$ are smooth rational.\\
ii) The first component $D_1$ is smooth of genus $g$ and $h|_{D_1}:D_1\to X$ is an embedding.\\
iii) The morphism $h$ is unramified.\\
iv) Two components $D_i$ and $D_j$ intersect transversally in one point if $|i-j|=1$
and not at all otherwise.

Then  $\km_g(X)$ is of dimension $g$ in $[h:D\to X]$.
\end{proposition}

\begin{proof}
We copy the argument from \cite[Lem.\ 2.7]{BHT}. First of all,
since $h$ is unramified, the complex $h^*\Omega_X\to\Omega_D$ is a locally free
sheaf of rank one concentrated in degree $-1$, the dual of which is denoted $\kn_h$.
Then, one proceeds by induction over $n$ and uses the exact sequence
$$0\to\kn_h(-x)|_{D'}\to \kn_h\to\kn_h|_{D_n}\to0,$$
where $D':=D_1\cup\ldots\cup D_{n-1}$ and $\{x\}=D_{n-1}\cap D_n$.
From the exact sequence $$0\to\kn_h^*|_{D_n}\to h^*\Omega_X|_{D_n}\to
\Omega_D|_{D_n}\to 0$$ and $\Omega_D|_{D_n}\cong\ko(-1)$,
one deduces  $\kn_h|_{D_n}\cong\ko(-1)$.
Thus, $H^i(\kn_h)\cong H^i(\kn_h(-x)|_{D'})$. On the other hand, $\kn_h(-x)|_{D'}=\kn_{h'}$,
where $h':=h|_{D'}:D'\to X$.  By induction this eventually yields  
$H^i(\kn_h)\cong H^i(\kn_{D_1/X})$. But clearly, $h^0(\kn_{D_1/X})=h^0(D_1,\omega_{D_1})=g$ and the deformations of $D_1\subset X$ are unobstructed.
\end{proof}

\begin{remark} Maybe more geometrically, the arguments show 
that deformations of $h:D\to X$ are all given by deforming $D_1\subset X$.
 \end{remark}
%%%%%%%%%%%%%%%%%%%%%%%%%%%%%%%%%%%% 
\section{Special elliptic surfaces}\label{sec:specialell}

We follow \cite[Sect.\ 4]{vGS} for the construction of an elliptic K3 surface
$X\to \PP^1$ with a symplectic involution given by a two-torsion section. Deformations of $X$ will lead to K3 surfaces with N\'eron--Severi group
$\widetilde\Lambda_{2d}$ with $d=2e>2$.

The elliptic K3 surface $X\to \PP^1$  is described by an equation
of the form \begin{equation}\label{eqn:ell}
y^2=x(x^2+a(t)x+b(t))
\end{equation} with general $a(t)$ and $b(t)$  of degree $4$ resp.\ $8$.
Then the fibration has two obvious sections: The section at infinity $\sigma$ given by  $x=z=0$, which will serve us as the zero section,  and a disjoint section $\tau$  given by $x=y=0$. Using the explicit equation, one finds that $\tau$
has order two. Thus, translation by $\tau$  defines an involution
$f:X\congpf X$ which is symplectic.

Still following \cite{vGS}, one computes the singular fibres of
$X\to\PP^1$: There are  eight fibres of type $I_1$
(a rational curve with one node)  and eight fibres of type $I_2$
(the union of two smooth $\PP^1$ intersecting transversally in two points). They can
be found over the zeroes of $b\in H^0(\PP^1,\ko(8))$ resp.\  $a^2-4b\in H^0(\PP^1,\ko(8))$.  The fixed points of $f$ are the nodes of the 
eight $I_1$-fibres which are all avoided by $\sigma$ and $\tau$.
 Moreover, $f$ interchanges the two components of
each $I_2$-fibre.

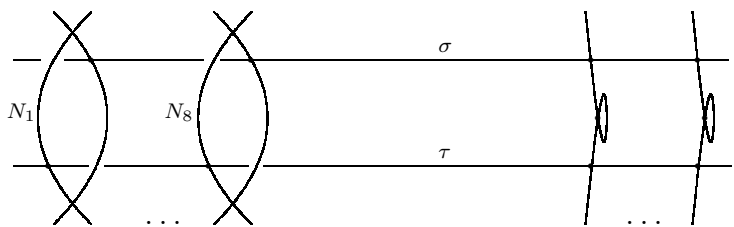
\vspace{-3ex}
$$
\begin{picture}(150,100)
\put(-50,80){\qbezier(0,0)(-40,-40)(0,-80)}
\put(-50,80){\qbezier(-14,0)(26,-40)(-14,-80)}
\put(-30,0){$\ldots$}
\put(-50,80){\qbezier(60,0)(20,-40)(60,-80)}
\put(-50,80){\qbezier(46,0)(86,-40)(46,-80)}

\put(80,80){\qbezier(55,0)(62,-70)(63,-40)}
\put(80,80){\qbezier(55,-80)(62,-10)(63,-40)}
\put(150,00){$\ldots$}
\put(120,80){\qbezier(55,0)(62,-70)(63,-40)}
\put(120,80){\qbezier(55,-80)(62,-10)(63,-40)}
\put(-1.5,62){\line(1,0){190}}
\put(-60,62){\line(1,0){52}}
\put(-79,62){\line(1,0){10}}
\put(-50.1,62){\circle*{2}}
\put(9.8,62){\circle*{2}}
\put(177,62){\circle*{2}}
\put(137,62){\circle*{2}}
\put(80,65){\tiny\mbox{$\sigma$}}
\put(14.5,22){\line(1,0){175}}
\put(-45,22){\line(1,0){54}}
\put(-79,22){\line(1,0){28}}
\put(-66.1,22.1){\circle*{2}}
\put(-6.1,22.1){\circle*{2}}
\put(177.1,22.1){\circle*{2}}
\put(137,22.1){\circle*{2}}
\put(139.6,40.1){\circle*{2}}
\put(179.6,40.1){\circle*{2}}
\put(80,25){\tiny\mbox{$\tau$}}
\put(-81.5,40){\tiny\mbox{$N_1$}}
\put(-22.5,40){\tiny\mbox{$N_8$}}
\end{picture}
$$
\vspace{2ex}

The components of the $I_2$-fibres not meeting $\sigma$ are denoted $N_1,\ldots,N_8$. Then $\hat N=(1/2)\sum N_i\in\NS(X)$. Moreover, 
if $F$ denotes the class of a generic fibre (and by abuse also 
a generic fibre itself), then $\sigma$ and $F$
span a hyperbolic plane and $\tau=\sigma+2F-\hat N$.
The N\'eron--Severi group of $X$ (for general $a$ and $b$)
is thus $\langle\sigma,F\rangle\oplus \langle N_1,\ldots,N_8,
\hat N\rangle$, which is of rank $10$.

Next consider a curve of the form $C=e N+F+\sigma+\tau$,
where $N$ is one of the $I_2$-fibres, and let $L:=\ko(C)$.
Then $L$ is big and nef. Indeed,
$(L.L)=4e>0$ and $C$ intersects all its irreducible components
positively, e.g.\ $(C.\sigma)=e-1>0$. In fact, $L$ is ample
as it clearly intersects all horizontal curves positively and has also positive intersection
with all $(-2)$-curves (e.g.\ the two components of the $I_2$-fibres).
Moreover, $L$  is primitive, as $(C.N_i)=1$.
Since $f$ respects the fibration and interchanges $\sigma$ and $\tau$, the curve
$C$ is $f$-invariant and disjoint from ${\rm Fix}(f)$.

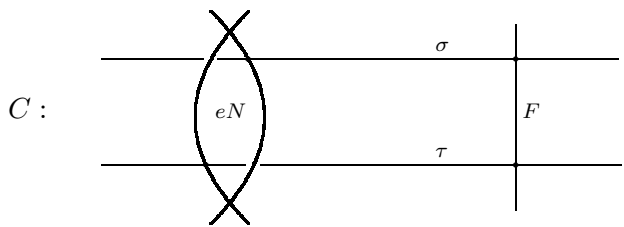
\vspace{-3ex}
$$
\begin{picture}(150,100)
{  \linethickness{0.28mm}\put(-50,80){\qbezier(60,0)(20,-40)(60,-80)}
\put(-50,80){\qbezier(46,0)(86,-40)(46,-80)}}
\put(110,5){\line(0,1){70}}
\put(-1.5,62){\line(1,0){150}}
\put(-45,62){\line(1,0){38}}
\put(9.9,62.){\circle*{2}}
\put(110,62.1){\circle*{2}}
\put(80,65){\tiny\mbox{$\sigma$}}
\put(14.5,22){\line(1,0){135}}
\put(-45,22){\line(1,0){54}}
\put(-5.9,22.1){\circle*{2}}
\put(110,22.1){\circle*{2}}
\put(80,25){\tiny\mbox{$\tau$}}
\put(-2.5,40){\tiny\mbox{$eN$}}
\put(112.5,40){\tiny\mbox{$F$}}
\put(-80,40){\small\mbox{$C:$}}
\end{picture}
$$
\vspace{-2ex}

Let us now consider the quotient $\bar X:=X/\langle f\rangle$ which is a singular
K3 surface with eight ordinary double points.
Its minimal resolution $Y\to \bar X$ comes with a natural elliptic fibration $Y\to \PP^1$. Note that the
fibres of type $I_1$ and $I_2$ are interchanged when passing from $X$ to $Y$.

The quotient $\bar C:=C/\langle f\rangle\subset \bar X$
avoids the singular locus of $\bar X$ and can thus also
be viewed  as a curve in $Y$. For the same reason, the line bundle $L$ descends
to an ample line bundle $\bar L$ on $\bar Y$.
Note that $\bar C$ decomposes as $\bar C=e\bar N+\bar F
+\bar\sigma$, where $\bar N$ is an $I_1$-fibre of $Y\to \PP^1$,
$\bar F$ is a smooth fibre, and $\bar\sigma$ is a section.

$$
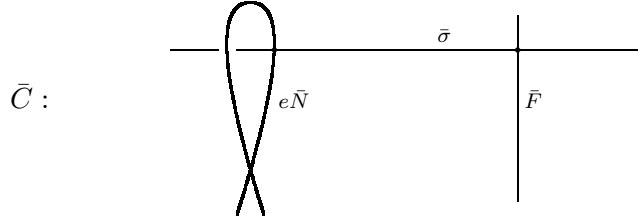
\begin{picture}(150,100)
{  \linethickness{0.28mm}
\put(-50,80){\qbezier(60,0)(40,-0)(65,-80)}
\put(-50,80){\qbezier(60,0)(80,-0)(55,-80)}
%\put(-50,80){\qbezier(46,0)(86,-40)(46,-80)}
}
\put(110,5){\line(0,1){70}}
\put(4.7,62){\line(1,0){150}}
\put(-20,62){\line(1,0){18}}
\put(18.9,62.1){\circle*{2}}
\put(110,62.1){\circle*{2}}
\put(80,65){\tiny\mbox{$\bar\sigma$}}
%\put(14.5,22){\line(1,0){135}}
%\put(-45,22){\line(1,0){54}}
%\put(-6.1,22.1){\circle*{2}}
%\put(110,22.1){\circle*{2}}
%\put(80,25){\tiny\mbox{$\tau$}}
\put(20.5,40){\tiny\mbox{$e\bar N$}}
\put(112.5,40){\tiny\mbox{$\bar F$}}
\put(-80,40){\small\mbox{$\bar C:$}}
\end{picture}
$$

\begin{lemma}
There exists a stable map $h:D\to \bar X$ of arithmetic genus one
with image $\bar C$ and such that $\km_1(\bar X,\bar L)$ is one-dimensional
in $h$.
\end{lemma}

\begin{proof}
Since $\bar C$ avoids the singularities of $\bar X$, we can equally work
with $\bar C\subset Y$. The curve $D$ shall have components $D_1,D_2,\ldots, D_n$,
$n=e+2$, with $D_1\congpf \bar F$, $D_2\congpf \bar\sigma$, and
$D_i\to \bar N$, $i\geq2$, being the normalization. The gluing is defined according to the 
picture (cf.\ \cite{BHT,Kem,LL}):

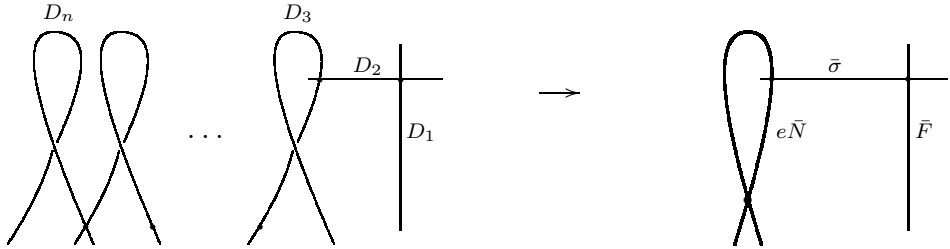
\vspace{-2ex}

$$
\begin{picture}(150,100)
\put(42,40){\tiny\mbox{$D_1$}}
\put(22,65){\tiny\mbox{$D_2$}}
\put(-4,85){\tiny\mbox{$D_3$}}
\put(-94,85){\tiny\mbox{$D_n$}}
\put(-150,80){\qbezier(60,0)(40,-0)(75,-80)}
\put(-150,80){\qbezier(60,0)(80,-0)(61.5,-41.4)}
\put(-150,80){\qbezier(59.6,-45)(56,-60)(43,-80)}
\put(-125,80){\qbezier(60,0)(40,-0)(75,-80)}
\put(-125,80){\qbezier(60,0)(80,-0)(61.5,-41.4)}
\put(-125,80){\qbezier(59.6,-45)(56,-60)(43,-80)}
\put(-77.8,6.4){\circle*{2}}
\put(9.8,62){\circle*{2}}
\put(40,62){\circle*{2}}
\put(-52.8,6.4){\circle*{2}}
\put(-40,40){\mbox{$\ldots$}}
\put(40,5){\line(0,1){70}}
\put(5.7,62){\line(1,0){50}}
\put(90,55){\tiny\mbox{$\to$}}
\put(-60,80){\qbezier(60,0)(40,-0)(75,-80)}
\put(-60,80){\qbezier(60,0)(80,-0)(61.5,-41.4)}
\put(-60,80){\qbezier(59.6,-45)(56,-60)(43,-80)}
\put(-12.7,6.4){\circle*{2}}
{  \linethickness{0.28mm}
\put(110,80){\qbezier(60,0)(40,-0)(65,-80)}
\put(110,80){\qbezier(60,0)(80,-0)(55,-80)}
\put(170,16.6){\circle*{3}}

%\put(-50,80){\qbezier(46,0)(86,-40)(46,-80)}
}
\put(230,5){\line(0,1){70}}
\put(174.7,62){\line(1,0){70}}
%\put(100,62){\line(1,0){18}}
\put(178.9,62.1){\circle*{2}}
\put(230,62.1){\circle*{2}}
\put(200,65){\tiny\mbox{$\bar\sigma$}}
\put(180.5,40){\tiny\mbox{$e\bar N$}}
\put(232.5,40){\tiny\mbox{$\bar F$}}
%\put(-80,40){\small\mbox{$\bar C$}}
\end{picture}
$$

Obviously, $D$ is of arithmetic genus one and $h_*(D)=\bar C$.
Moreover, the assumptions of Proposition \ref{prop:defounramified} are satisfied and hence
$\km_1(Y)$ is of dimension one in $[h:D\to  Y]$.
\end{proof}

Now consider a generic deformation \begin{equation}\label{eqn:fam}
(\kx,\kl)\to S\end{equation}  of $(X,L,f)$, i.e.\
$(\kx_0,f_0)=(X,f)$ for a distinguished $0\in S$ and for generic $t\in S$
the fibre $\NS(\kx_t)$ has  rank $\rho=9$ with $f_t$-invariant part spanned by
$\kl_t$.  Taking quotients, one obtains a family of singular K3 surfaces
$\bar\kx\to S$. Clearly, $L=\ko(C)$ descends to the quotient $\bar X$, for
$C$ is $f$-invariant and avoids the fixed points. Hence,  also the line bundle
$\kl$ descends to  a relative ample line bundle $\bar\kl$. (The obstructions to
deform $L$ resp.\ $\bar L$ sideways are the same.)
Note that for generic $t\in S$ the line bundle $\bar\kl_t$ generates $\Pic(\bar\kx_t)$.

Let us apply the discussion of Section \ref{sec:defo} to $h:D\to\bar\kx_0=\bar X$. So, we
consider the
relative moduli space of stable maps of genus one $\km_1(\bar\kx,\bar\kl)\to
S$.

\begin{corollary}\label{cor:generic}
The stable map $h:D\to\bar X$ thus constructed deforms sideways to stable maps $h_t:D_t\to\bar\kx_t$. Moreover, for generic $t\in S$ the curve $h_{t*}(D_t)\subset\bar\kx_t$ 
and its preimage in $\kx_t$ are  integral and disjoint from the singular locus
resp.\ the fixed point set of $f_t$.
\end{corollary}

\begin{proof}
The existence of the deformation to the nearby fibres follows directly from Proposition
\ref{prop:defounramified} and Corollary \ref{cor:defo}. Since $\bar C=h_*(D)$ avoids the singularities of $\bar X$, this will
hold for generic $t$. Clearly, $h_{t*}(D_t)\in|\bar\kl_t|$. Therefore, since
$\kl_t$ generates  the invariant part of $\NS(\kx_t)$ and hence $\bar\kl_t$ generates  $\NS(\bar\kx_t)$, the curve
$h_{t*}(D_t)$ must be integral.

Suppose the preimage $C_t$ of $h_{t*}(D_t)$ were not integral for $t$ generic, i.e.\  $C_t=C_t'+C_t''$ with $f_t(C_t')=C_t''$. (Use that $f_t$ is an involution.) The two components would then specialize to
$C'$ resp.\ $C''$ on $X$ with  $C=C'+C''$ and $f(C')=C''$. We may assume
that $F\subset C'$. But then also $F=f(F)\subset f(C')=C''$ which eventually yields
the contradiction that $F$ appears with multiplicity at least two in $C$. 
\end{proof}

\begin{remark}\label{rem:covering} In fact, since the stable map $h:D\to X$ deforms with the fibre component $F$ in a one-dimensional family, also the deformations $h_t:D_t\to \bar\kx_t$
 come in a  family dominating $\bar\kx_t$. 
 %Alternatively, one could argue that the fibres
%of $\km_1(\bar\kx,\bar\kl)\to S$ are of dimension at least one.
Thus, one obtains a dominating family of integral genus one curves in the generic
deformation $\kx_t$ that are $f_t$-invariant and avoid the fixed points of $f_t$.
\end{remark}
%%%%%%%%%%%%%%%%%%%%%%%%%%%%%%%%%%%%
\section{Proof of the main theorem}

The outcome of the above construction are generic
K3 surfaces $\kx_t\in\gm_{\widetilde\Lambda_d}$ with a symplectic involution $f_t$ such that $\bar\kx_t=\kx_t/\langle f_t\rangle$ contains a one-dimensional family of integral curves of geometric genus one that avoids the singular locus. 

This immediately leads to a proof of our main result.

\begin{theorem}
For all $(X,f)\in\gm_{\widetilde\Lambda}$, the symplectic involution
$f:X\congpf X$ acts  as $\id$ on $\CH^2(X)$.
\end{theorem} 

\begin{proof} The case $d=2$ follows from \cite{Ped}. So we assume $d=2e>2$, i.e.\ $e>1$. 
We first show that the above discussion combined with Proposition \ref{prop:Roit} proves the assertion for generic $(X,f)\in\gm_{\widetilde\Lambda_{d}}$. 

Consider a deformation (\ref{eqn:fam})  of the special elliptic K3 surface (\ref{eqn:ell}).
Then for generic $t\in S$ one has $\NS(\kx_t)=\widetilde\Lambda_d$. Indeed, by \cite[Prop.\ 2.7]{vGS} only in this case all $f_t$-invariant line bundles actually descend to the quotient $\bar\kx_t$. 
Hence, the elliptic K3 surfaces described by (\ref{eqn:ell}) can be connected to the generic K3 surface parametrized by $\gm_{\widetilde\Lambda_d}$. Here, we use 
 that $\gm_{\widetilde\Lambda_d}$ is connected.
 
 The generic fibre of the family (\ref{eqn:fam}) satisfies the assumption of Proposition \ref{prop:Roit}.
 Indeed, by Corollary \ref{cor:generic} and Remark \ref{rem:covering} there exists a dominating family
 of integral curves of arithmetic genus one on the generic fibre $\kx_t$ that are invariant under the
 involution and avoid the fixed points.

Now consider an arbitrary $(X,f)\in\gm_{\widetilde\Lambda_d}$. Then any $x\in X$ can be viewed
as a specialization of points $x_t$ in generic deformations $(\kx_t,f_t)\in\gm_{\widetilde\Lambda_d}$. Clearly, the points $f_t(x_t)$ then specialize to  $f(x)$. For generic $\kx_t$ we
have proved $[x_t]=[f_t(x_t)]$ in $\CH^2(\kx_t)$ already and specialization thus yields $[x]=[f(x)]$
in $\CH^2(X)$ for all $x\in X$.
\end{proof}
 
%%%%%%%%%%%%%%%%%%%%%%%%%%%%%%%%%%%%%

\section{Further comments} \label{sec:p357}
 
We briefly outline how to adapt our techniques to the case of symplectic automorphisms of prime
order. For $p=3,5$, and $7$, Garbagnati and Sarti describe in \cite[Thm.\ 4.1]{GS}  lattices $\Omega_p$
of rank $12$, $16$, resp.\ $18$ that are isomorphic to the anti-invariant part of $f^*$ acting
on $H^2(X,\ZZ)$. Similar to the case $p=2$, the generic polarized K3 surface $(X,L)$ of degree $2d$ with a symplectic automorphism $f:X\congpf X$ of order $p$ leaving $L$ fixed
has Picard group isomorphic to
$\Lambda_{p,d}:=\ZZ L\oplus\Omega_p$ or possibly, if  $d\equiv 0(p)$, 
isomorphic to  a  lattice  $\widetilde\Lambda_{p,d}$
that contains $\Lambda_{p,d}$ as a primitive  sublattice of index $p$.
In fact, the case $\Lambda_{7,d}$ is not realized if $d\equiv 0(7)$ 
(cf.\ \cite[Prop.\ 5.2]{GS}), but unfortunately it is not known whether
the lattices $\widetilde\Lambda_{p,d}$ are unique for given $p$ and $d\equiv0(p)$
(see \cite[Sec.\ 6]{GS}).
 The moduli spaces are of dimension $7$, $3$, resp.\ $1$.

Examples of symplectic automorphisms of order $3,5$, and $7$ have been described
in \cite[Sec.\ 3.1]{GS}.
They are again given by translation by a torsion section. The Picard numbers in these examples
are $14, 18$, resp.\ $20$ and in each case they correspond to points in (at least) one of
component of the moduli space
of polarized K3 surfaces $(X,L)$ with a symplectic automorphism $f$ of degree $L^2=2d$.
This leads to the following result:

\begin{theorem}\label{thm:manin7}  For $p=3,5$, or $7$ and $d=ep$, there exists
one component of the moduli space of polarized K3 surfaces $(X,L)$ with a symplectic automorphism $f:X\congpf X$ of order $p$ and $L^2=2d$  such that Conjecture \ref{conj:BB}
holds true.\qqed
\end{theorem}

 It is very likely that for $p=7$ and $d\equiv 0(7)$ the result can be strenghtened
 to cover all K3 surfaces, as we would expect that $\widetilde\Lambda_{7,d}$ is in fact
 unique.
 
%%%%%%%%%%%%%%%%%%%%%%%%%%%%%%%%%%%%% 

%CHECK ELKIES LN for construciton of ell

\end{document}